\documentclass[final]{dmtcs-episciences}
\usepackage[utf8]{inputenc}
\usepackage[numbers]{natbib}
\usepackage[english]{babel}

\usepackage{listings}
\usepackage{enumerate}
\usepackage{layout}
\usepackage{verbatim}

\usepackage{amsmath,amsfonts,amssymb, amsthm}
\usepackage[mathscr]{euscript}
\mathchardef\mhyphen="2D 

\newtheorem{theorem}{Theorem}[section]
\newtheorem{proposition}[theorem]{Proposition}
\newtheorem{lemma}[theorem]{Lemma}
\newtheorem{corollary}[theorem]{Corollary}
\newtheorem{observation}[theorem]{Observation}

\theoremstyle{definition}
\newtheorem{definition}[theorem]{Definition}

\newcommand{\LR}{\textsc{lr}}
\newcommand{\LRcl}[1]{#1^{\LR}} 
\newcommand{\LRdot}{\odot_{\text{\LR}}} 
\newcommand{\Av}[1]{\text{Av}(#1)}

\newcommand{\includefigure}[2][scale=0.8]{\includegraphics[#1]{figure-#2}}
\newcommand{\mathfigure}[2][scale=0.8]{\vcenter{\hbox{\includefigure[#1]{#2}}}}
\newcommand{\centeredfigure}[2][scale=0.8]{$\mathfigure[#1]{#2}$}

\title{Splittability and 1-Amalgamability of Permutation Classes}
\keywords{permutation classes, splittability, amalgamability, 1-amalgamability}
\author{
Vít Jelínek\affiliationmark{1}\thanks{Supported by the project 16-01602Y of the Czech Science
Foundation.}
\and Michal Opler\affiliationmark{1}\thanks{This work has received financial support from
the Neuron Foundation for Support of Science.}
}
\affiliation{Computer Science Institute, Charles University in Prague}

\received{2017-5-1}
\revised{2017-9-8}
\accepted{2017-11-21}
\publicationdetails{19}{2017}{2}{4}{3292}

\begin{document}
\maketitle

\begin{abstract}
A permutation class $C$ is splittable if it is contained in a merge of two of its proper subclasses,
and it is 1-amalgamable if given two permutations $\sigma, \tau \in C$, each with a marked element,
we can find a permutation $\pi \in C$ containing both $\sigma$ and $\tau$ such that the two marked
elements coincide. It was previously shown that unsplittability implies 1-amalgamability. We prove
that unsplittability and 1-amalgamability are not equivalent properties of permutation classes by
showing that the class $\Av{1423, 1342}$ is both splittable and 1-amalgamable. Our construction is
based on the concept of LR-inflations, which we introduce here and which may be of independent
interest.
\end{abstract}

\section{Introduction}

In the study of permutation classes, a notable interest has recently been directed towards the
operation of merging. We say that a permutation $\pi$ is a \emph{merge} of $\sigma$ and $\tau$ if
the elements of $\pi$ can be colored red and blue so that the red elements form a copy of $\sigma$
and the blue elements form a copy of $\tau$. For instance, Claesson, Jelínek and Steingrímsson
\cite{Claesson2012} showed that every 1324-avoiding permutation can be merged from a 132-avoiding
permutation and a 213-avoiding permutation, and used this fact to prove that there are at most
$16^n$ 1324-avoiding permutations of length $n$.

A general problem that follows naturally is how to identify when a permutation class $C$ has proper
subclasses $A$ and $B$, such that every element of $C$ can be obtained as a merge of an element of
$A$ and an element of $B$. We say that such a permutation class $C$ is \emph{splittable}. Jelínek
and Valtr \cite{Jelinek2015} showed that every inflation-closed class is unsplittable and the class
of $\sigma$-avoiding permutations, where $\sigma$ is a direct sum of two nonempty permutations and
has length at least four, is splittable. Furthermore, they mentioned the connection of splittability
to more general structural properties of classes of relational structures studied in the area of
Ramsey theory, most notably the notion of 1-amalgamability. We say that a permutation class $C$ is
\emph{1-amalgamable} if given two permutations $\sigma, \tau \in C$, each with a marked element, we
can find a permutation $\pi \in C$ containing both $\sigma$ and $\tau$ such that the two marked
elements coincide.

Not much is known about 1-amalgamability of permutation classes. Jelínek and Valtr \citep[Lemma
1.5]{Jelinek2015}, using a more general result from Ramsey theory, showed that unsplittability
implies 1-amalgamability, and they raised the question whether there is a permutation class that is
both splittable and 1-amalgamable. In this paper, we answer this question by showing that the class
$\Av{1423, 1342}$ has both properties.

For this task, we will introduce a slightly weaker property than being inflation-closed, that is
being closed under inflating just the elements that are left-to-right minima. We say that an element
of permutation $\pi$ is a \emph{left-to-right minimum}, or just LR-minimum, if it is smaller than
all the elements preceding it. In Section \ref{sec:LR inflations} we shall prove that certain
properties of a permutation class $C$ imply that its closure under inflating LR-minima is splittable
and 1-amalgamable. And finally in Section \ref{sec:main result} we show that the class $\Av{1423,
1342}$ is actually equal to the class $\Av{123}$ closed under inflating left-to-right minima and
that $\Av{123}$ has the desired properties.

\section{Basics}
A \emph{permutation} $\pi$ of length $n \geq 1$ is a sequence of all the $n$ distinct numbers from
the set $[n] = \lbrace 1, 2, \ldots, n \rbrace$. We denote the $i$-th element of $\pi$ as $\pi_i$.
Note that we omit all punctuation when writing out short permutations, e.g., we write 123 instead of
1, 2, 3. The set of all permutations of length $n$ is denoted $S_n$.

We say that two sequences of distinct numbers $a_1, \ldots, a_n$ and $b_1, \ldots, b_n$ are
\emph{order-isomorphic} if for every two indices $i < j$ we have $a_i < a_j$ if and only if $b_i <
b_j$. Given two permutations $\pi \in S_n$ and $\sigma \in S_k$, we say that $\pi$ \emph{contains}
$\sigma$ if there is a $k$-tuple $1 \leq i_1 < i_2 < \cdots < i_k \leq n$ such that the sequence
$\pi_{i_1}, \pi_{i_2}, \ldots, \pi_{i_k}$ is order-isomorphic to $\sigma$ and we say that such a
sequence is an \emph{occurrence} of $\sigma$ in $\pi$. Furthermore, we say that the corresponding
function $f: [k] \to [n]$ defined as $f(j) = i_j$ is an \emph{embedding} of $\sigma$ into $\pi$. In
the context of permutation containment, we often refer to the permutation $\sigma$ as a
\emph{pattern}.

A permutation that does not contain $\sigma$ is \emph{$\sigma$-avoiding} and we let $\Av{\sigma}$
denote the set of all $\sigma$-avoiding permutations. Similarly, for a set of permutations $F$ , we
let $\Av{F}$ denote the set of permutations that avoid all elements of $F$. Note that for small sets
$F$ we omit the curly braces, e.g., we simply write $\Av{\sigma, \rho}$ instead of $\Av{\lbrace
\sigma, \rho \rbrace}$.

We say that a set of permutations $C$ is a \emph{permutation class} if for every $\pi\in C$ and
$\sigma$ contained in $\pi$, $\sigma$ belongs to $C$ as well. Observe that a set of permutations $C$
is a permutation class if and only if there is a set $F$ such that $C = \Av{F}$. Moreover, for every
permutation class $C$, there is a unique inclusionwise minimal set $F$ such that $C=\Av{F}$; this
set $F$ is known as the \emph{basis} of~$C$. A class is said to be \emph{principal} if its basis has
a single element, i.e., if the class has the form $\Av{\sigma}$ for a permutation~$\sigma$.

Suppose that $\pi \in S_n$ is a permutation, let $\sigma_1, \ldots, \sigma_n$ be an $n$-tuple of
non-empty permutations, and let $m_i$ be the length of $\sigma_i$ for $i \in [n]$. The
\emph{inflation} of $\pi$ by the sequence $\sigma_1, \ldots, \sigma_n$, denoted by $\pi[\sigma_1,
\ldots, \sigma_n]$, is the permutation of length $m_1 +\cdots+ m_n$ obtained by concatenating $n$
sequences $\overline{\sigma}_1 \overline{\sigma}_2 \cdots \overline{\sigma}_n$ with these
properties:
\begin{itemize}
  \item for each $i \in [n]$, $\overline{\sigma}_i$ is order-isomorphic to $\sigma_i$, and
  \item for each $i, j \in [n]$, if $\pi_i < \pi_j$, then all the elements of $\overline{\sigma}_i$
  are smaller than all the elements of $\overline{\sigma}_j$.
\end{itemize}

\begin{figure}[h!]
\centering
\begin{displaymath}
\mathfigure{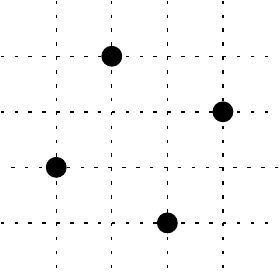}\left[\mathfigure{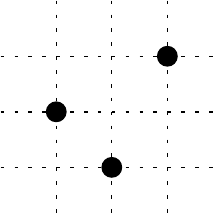}, \mathfigure{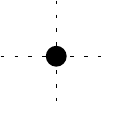}, \mathfigure{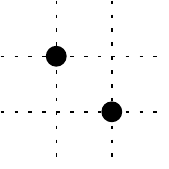},
\mathfigure{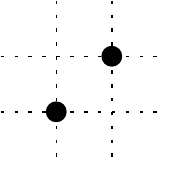}\right] = \mathfigure{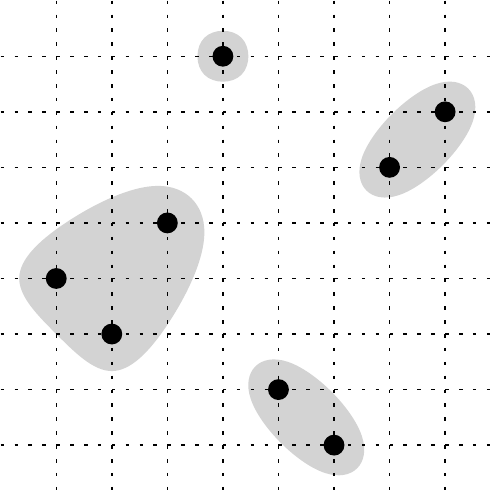}
\end{displaymath}
\caption{An example of inflation: $2413[213, 1, 21, 12] = 43582167$.}
\label{fig:inflation}
\end{figure}

For two sets of permutations $A$ and $B$, we let $A[B]$ denote the set of all the permutations that
can be obtained as  an inflation of a permutation from $A$ by a sequence of permutations from $B$.
We say that a set of permutations $A$ is \emph{$\cdot[B]$-closed} if $A[B] \subseteq A$, and
similarly a set of permutations $B$ is \emph{$A[\cdot]$-closed} if $A[B] \subseteq B$. Finally, we
say that a set of permutations $C$ is \emph{inflation-closed} if $C[C] \subseteq C$.

There is a nice way to characterize an inflation-closed class through its basis. We say that a
permutation $\pi$ is \emph{simple} if it cannot be obtained by inflation from smaller permutations,
except for the trivial inflations $\pi[1, \ldots, 1]$ and $1[\pi]$. Inflation-closed permutation
classes are precisely the classes whose basis only contains simple permutations \citep[Proposition
1]{Albert2005}.

\section{Splittability and 1-amalgamability}
\label{sec:split and am}
We now focus on the properties of splittability and 1-amalgamability of permutation classes. Mostly,
we state or rephrase results that were already known. For more detailed overview, especially
regarding splittability, see Jelínek and Valtr \cite{Jelinek2015}.

\subsection{Splittability}
\label{sec:splittability}
We say that  a permutation $\pi$ is a \emph{merge} of permutations $\tau$ and $\sigma$, if it can be
partitioned into two disjoint subsequences, one of which is an occurrence of $\sigma$ and the other
is an occurrence of $\tau$. For two permutation classes $A$ and $B$, we write $A \odot B$ for the
class of all merges of a (possibly empty) permutation from A with a (possibly empty) permutation
from B. Trivially, $A \odot B$ is again a permutation class.

Conversely, we say that a multiset of permutation classes $\lbrace P_1, \ldots, P_m\rbrace$ forms a
\emph{splitting} of a permutation class $C$ if $C \subseteq P_1 \odot \cdots \odot P_m$. We call
$P_i$ the \emph{parts} of the splitting. The splitting is \emph{nontrivial} if none of its parts is
a superset of $C$, and the splitting is \emph{irredundant} if no proper submultiset of $\lbrace P_1,
\ldots, P_m\rbrace$ forms a splitting of $C$. A permutation class $C$ is then \emph{splittable} if
$C$ admits a nontrivial splitting.

The following simple lemma is due to Jelínek and Valtr \citep[Lemma 1.3]{Jelinek2015}.

\begin{lemma}
\label{lemma:splittable conditions}
For a class $C$ of permutations, the following
properties are equivalent:
\begin{enumerate}[(a)]
  \item $C$ is splittable.
  \item $C$ has a nontrivial splitting into two parts.
  \item $C$ has a splitting into two parts, in which each part is a proper subclass of $C$.
  \item $C$ has a nontrivial splitting into two parts, in which each part is a principal class.
\end{enumerate}
\end{lemma}

Following the previous Lemma~\ref{lemma:splittable conditions}, we can characterize a splittable
class $C$ by the splittings of the form $\lbrace \Av{\pi}, \Av{\sigma} \rbrace$, where both $\pi$
and $\sigma$ are permutations from $C$. We want to identify permutations inside $C$ that cannot
define any such splitting.

\begin{definition}
\label{def:unavoidable}
Let $C$ be a permutation class. We say that a permutation $\pi \in C$ is \emph{unavoidable in $C$},
if for any permutation $\tau \in C$, there is a permutation $\sigma \in C$ such that any red-blue
coloring of $\sigma$ has a red copy of $\tau$ or a blue copy of $\pi$. We let $U_C$ denote the set
of all unavoidable permutations in $C$.
\end{definition}

It is easy to see that a permutation $\pi$ is unavoidable in $C$ if and only if $C$ has no
nontrivial splitting into two parts with one part being $\Av{\pi}$. A more detailed overview of the
properties of unavoidable permutations was provided by Jelínek and Valtr~\citep[Observation
2.2-3]{Jelinek2015}, we will mention only the observations needed for our results.

Note that for a nonempty permutation class $C$, the set of unavoidable permutations $U_C$ is in fact
a nonempty permutation class contained in the class $C$. We can use the class of unavoidable
permutations to characterize the unsplittable permutation classes.

\begin{observation}
\label{lemmma:UC=C}
A permutation class $C$ is unsplittable if and only if $U_C = C$.
\end{observation}

Furthermore, we can show that if $C$ is closed under certain inflations then also $U_C$ is closed
under the same inflations. Again, the following result is due to Jelínek and Valtr~\citep[Lemma
2.4]{Jelinek2015}.

\begin{lemma}
Let $C$ be a permutation class. If, for a set of permutations $X$, the class $C$ is closed under
$\cdot[X]$, then $U_C$ is also closed under $\cdot[X]$, and if $C$ is closed under $X[\cdot]$, then
so is $U_C$. Consequently, if $C$ is inflation-closed, then $U_C = C$ and $C$ is unsplittable.
\end{lemma}

\subsection{Amalgamability}
\label{sec:amalgability}

Now let us introduce the concept of amalgamation, which comes from the general study of relational
structures.

We say that a permutation class $C$ is \emph{$\pi$-amalgamable} if for any two permutations $\tau_1,
\tau_2 \in C$ and any two mappings $f_1$ and $f_2$, where $f_i$ is an embedding of $\pi$ into
$\tau_i$, there is a permutation $\sigma \in C$ and two mappings $g_1$ and $g_2$ such that $g_i$ is
an embedding of $\tau_i$ into $\sigma$ and $g_1 \circ f_1 = g_2 \circ f_2$. We also say, for $k \in
\mathbb{N}$ that a permutation class $C$ is $k$-amalgamable if it is $\pi$-amalgamable for every
$\pi$ of order at most $k$. Furthermore, a permutation class $C$ is amalgamable if it is
$k$-amalgamable for every $k$.

\begin{figure}[h!]
\centering
\begin{displaymath}
\mathfigure{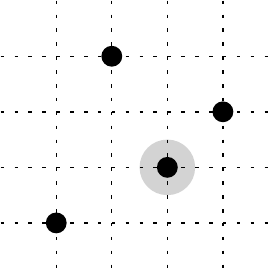} \qquad \mathfigure{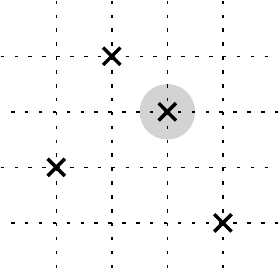} \quad \longrightarrow \quad
\mathfigure{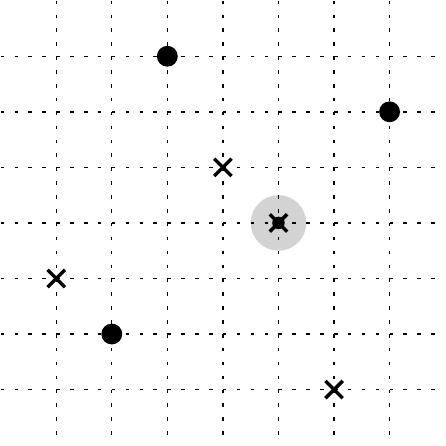}
\end{displaymath}
\caption{One possible 1-amalgamation of 1423 and 2431 with highlighted embeddings of the singleton
permutations is the permutation 3275416.}
\label{fig:1-amalgamation}
\end{figure}

Note that $k$-amalgamability implies $(k-1)$-amalgamability, so we have an infinite number of
increasingly stronger properties. However, the situation is quite simple in the case of the
permutation classes. As shown by Cameron~\citep{Cameron2002}, there are only five infinite
amalgamable classes, the classes $\Av{12}$, $\Av{21}$, the class of all layered permutations
$\Av{231, 312}$, the class of their complements $\Av{213, 132}$ and the class of all permutations.
These are also the only permutation classes that are 3-amalgamable, implying that for any $k\ge 3$,
a permutation class is $k$-amalgamable if and only if it is amalgamable.

In contrast, very little is known about 1-amalgamable and 2-amalgamable permutation classes. In this
paper, we are particularly interested in the 1-amalgamable permutation classes.

\begin{definition}
\label{def:1-amalgamable}
Let $C$ be a permutation class. We say that a permutation $\pi \in C$ is \emph{1-amalgamable in
$C$}, if for every $\tau \in C$ and every prescribed pair of embeddings $f_1$ and $f_2$ of the
singleton permutation 1 into $\pi$ and $\tau$ there is a permutation $\sigma \in C$ and embeddings
$g_1$ and $g_2$ of $\pi$ and $\tau$ into $\sigma$ such that $g_1 \circ f_1 = g_2 \circ f_2$. We use
$A_C$ to denote the set of all 1-amalgamable permutations in $C$.
\end{definition}

Trivially, $A_C$ is a permutation class contained in $C$. Moreover, the properties of $A_C$ are
largely analogous to those of $U_C$, as shown by the next several results.

\begin{observation}
A permutation class $C$ is 1-amalgamable if and only if $A_C=C$.
\end{observation}

Similarly to $U_C$, the set $A_C$ is closed under the same inflations as the original class~$C$.

\begin{lemma}
Let $C$ be a permutation class. If, for a set of permutations $X$, the class $C$ is closed under
$\cdot[X]$, then $A_C$ is also closed under $\cdot[X]$, and if $C$ is closed under $X[\cdot]$, then
so is $A_C$. Consequently, if $C$ is inflation-closed, then $A_C = C$ and $C$ is 1-amalgamable.
\end{lemma}
\begin{proof}
Suppose that $C$ is closed under $\cdot[X]$. We can assume that $X$ itself is inflation-closed since
if $C$ is closed under $\cdot[X]$, it is also closed under $\cdot[X[X]]$.

Let $\pi \in A_C$ be a 1-amalgamable permutation of order $k$ and let $\rho_1, \ldots, \rho_k$ be
permutations from $X$. Our goal is to prove that $\pi[\rho_1, \ldots, \rho_k]$ also belongs to
$A_C$. We can assume, without loss of generality, that all $\rho_i$ are actually equal to a single
permutation~$\rho$. Otherwise, we could just take $\rho \in X$ that contains every $\rho_i$ (this is
possible since $X$ is inflation-closed) and prove the stronger claim that $\pi[\rho, \ldots, \rho]$
belongs to $A_C$. Let us use $\pi[\rho]$ as a shorthand notation for $\pi[\rho, \ldots, \rho]$.

It is now sufficient to show that $\pi[\rho]$ belongs to $A_C$ for every $\pi \in A_C$ and $\rho \in
X$. Fix a permutation $\tau \in C$ and two embeddings $f_1$ and $f_2$ of the singleton permutation
into $\pi[\rho]$ and $\tau$. We aim to find a permutation $\sigma \in C$ and two embeddings $g_1$
and $g_2$ of $\pi[\rho]$ and $\tau$ into $\sigma$ such that $g_1 \circ f_1 = g_2 \circ f_2$. We can
straightforwardly decompose $f_1$ into an embedding $h_1$ of the singleton permutation into $\pi$,
by simply looking to which inflated block order-isomorphic to $\rho$ the image of $f_1$ belongs, and
an embedding $h_2$ of the singleton permutation into $\rho$, determined by restricting $f_1$ only to
that copy of $\rho$. Since $\pi$ belongs to $A_C$, there is a permutation $\sigma'$ with embeddings
$g_1'$ and $g_2'$ of $\pi$ and $\tau$ such that $g_1' \circ h_1 = g_2' \circ f_2$.

Define $\sigma = \sigma'[\rho]$, and view $\sigma$ as a concatenation of blocks, each a copy of
$\rho$. Let us define mapping $g_1$ by simply using $g_1'$ to map blocks of $\pi[\rho]$ to the
blocks of $\sigma$, each element in $\pi[\rho]$ gets mapped to the same element of the corresponding
copy of $\rho$ in $\sigma$. Then define mapping $g_2$ by using $g_2'$ to map its elements to the
blocks of $\sigma$ and then within the copy of $\rho$ to the single element in the image of $h_2$.
It is easy to see that $g_1$ and $g_2$ are in fact embeddings of $\pi[\rho]$ and $\tau$ into
$\sigma$. Also the images of $g_1 \circ f_1$ and $g_2 \circ f_2$ must lie in the same block of
$\sigma$. And finally these images must be equal since we used $h_2$ to place the single element
from the image of $g_2$ inside each block of $\sigma$.

We now show that if $C$ is closed under $X[\cdot]$ then so is $A_C$. Fix a permutation $\rho \in X$
of order $k$, and a $k$-tuple $\pi_1, \ldots, \pi_k$ of permutations from $A_C$. We will show that
$\rho[\pi_1, \ldots , \pi_k]$ belongs to $A_C$.

Fix a permutation $\tau \in C$ and two embeddings $f_1$ and $f_2$ of the singleton permutation into
$\rho[\pi_1, \ldots , \pi_k]$ and $\tau$. We aim to find a permutation $\sigma \in C$ and two
embeddings $g_1$ and $g_2$ of $\rho[\pi_1, \ldots , \pi_k]$ and $\tau$ into $\sigma$ such that $g_1
\circ f_1 = g_2 \circ f_2$. We again view $\rho[\pi_1, \ldots , \pi_k]$ as a concatenation of $k$
blocks, the $i$-th block being order-isomorphic to $\pi_i$. Suppose that the image of $f_1$ is in
the $j$-th block. Let us decompose $f_1$ into an embedding $h_1$ of the singleton permutation into
$\rho$ whose image is the $j$-th element of $\rho$, and an embedding $h_2$ of the singleton
permutation into $\pi_j$. Since $\pi_j$ belongs to $A_C$, there is a permutation $\sigma'$ with
embeddings $g_1'$ and $g_2'$ of $\pi_j$ and $\tau$ such that $g_1' \circ h_2 = g_2' \circ f_2$.

Define $\sigma = \rho[\pi_1, \ldots, \pi_{j-1} , \sigma', \pi_{j+1}, \ldots, \pi_k]$ and let us
define mapping $g_1$ in the following way. Every block of $\rho[\pi_1, \ldots , \pi_k]$ except for
the $j$-th one gets mapped to the corresponding block of $\sigma$, and the $j$-th block is mapped
using the embedding $g_1'$ to the $j$-th block of $\sigma$. Then define mapping $g_2$ simply by
mapping $\tau$ to the $j$-th block of $\sigma$ using $g_2'$. It is easy to see that both $g_1$ and
$g_2$ are in fact embeddings of $\rho[\pi_1, \ldots , \pi_k]$ and $\tau$ into $\sigma$. Furthermore,
the images of $g_1 \circ f_1$ and $g_2 \circ f_2$ both lie in the $j$-th block of $\sigma$. Their
equality then follows from the construction since $g_1' \circ h_2 = g_2' \circ f_2$.

It remains to show that if $C$ is inflation-closed then $A_C = C$. But if $C$ is inflation-closed,
then it is closed under $\cdot[C]$, so $A_C$ is also closed under $\cdot[C]$. And since $A_C$
trivially contains the singleton permutation, for every $\pi \in C$ we have that $\pi = 1[\pi]$ also
belongs to~$A_C$.
\end{proof}

As noted by Jelínek and Valtr~\citep[Lemma 1.5]{Jelinek2015}, it follows from the results of
Nešetřil~\citep{Nesetril1989} that if a permutation class $C$ is unsplittable then $C$ is also
1-amalgamable. Using the same argument, we get the following stronger proposition relating the
classes $U_C$ and $A_C$.

\begin{proposition}
Let $C$ be a permutation class, then $U_C \subseteq A_C$.
\end{proposition}
\begin{proof}
Let $\pi$ be an unavoidable permutation in $C$ and let $\tau$ be a permutation from $C$.  By the
definition of $U_C$, there is a permutation $\sigma \in C$ such that any red-blue coloring of
$\sigma$ has a red copy of $\tau$ or a blue copy of $\pi$.  We claim that $\sigma$ contains every
1-amalgamation of $\pi$ and $\tau$.  Suppose for a contradiction that there are two embeddings $f_1$
and $f_2$ of the singleton permutation 1 into $\pi$ and $\tau$ such that there are no embeddings
$g_1$ and $g_2$ of $\pi$ and $\tau$ into $\sigma$ that would satisfy $g_1 \circ f_1 = g_2 \circ
f_2$.

Let $f_1(1) = a$ and $f_2(1) = b$.  We aim to color the elements of $\sigma$ to avoid both a red
copy of $\tau$ and a blue copy of $\pi$.  We color an element $\sigma_i$ red if and only if there is
an embedding of $\pi$ which maps $\pi_a$ to $\sigma_i$.  Trivially, we cannot obtain a blue copy of
$\pi$, since we must have colored the image of $\pi_a$ red.  On the other hand, suppose we obtained
a red copy of $\tau$.  Then the image of $\tau_b$ was painted red which means that there is an
embedding of $\pi$ which maps $\pi_a$ to the same element.  We assumed that such a pair of
embeddings does not exist, therefore we defined a coloring of $\sigma$ that contains neither a red
copy of $\tau$ nor a blue copy of $\pi$.
\end{proof}

\section{Left-to-right minima}
\label{sec:LR inflations}
We say that the element $\pi_i$ \emph{covers the element} $\pi_j$ if $i < j$ and simultaneously
$\pi_i < \pi_j$. The $i$-th element of a permutation $\pi$ is then a \emph{left-to-right minimum},
or shortly LR-minimum, if it is not covered by any other element.

Similarly we could define LR-maxima, RL-minima and RL-maxima.  However we can easily translate
between right-to-left and left-to-right orientation by looking at the reverses of the permutations,
and similarly between maxima and minima by looking at the complements of the permutations. Therefore
we restrict ourselves to dealing only with LR-minima from now on.

\begin{definition}
Suppose that $\pi \in S_n$ is a permutation with $k$ LR-minima and let $\sigma_1, \ldots, \sigma_k$
be a $k$-tuple of non-empty permutations. The \emph{LR-inflation} of $\pi$ by the sequence
$\sigma_1, \ldots, \sigma_k$ is the permutation resulting from the inflation of the LR-minima of
$\pi$ by $\sigma_1, \ldots, \sigma_k$. We denote this by $\pi\langle\sigma_1, \ldots,
\sigma_k\rangle$.
\end{definition}

\begin{figure}[h!]
\centering
\begin{displaymath}
\mathfigure{perm-2413}\left \langle \mathfigure{perm-213}, \mathfigure{perm-21} \right\rangle =
\mathfigure{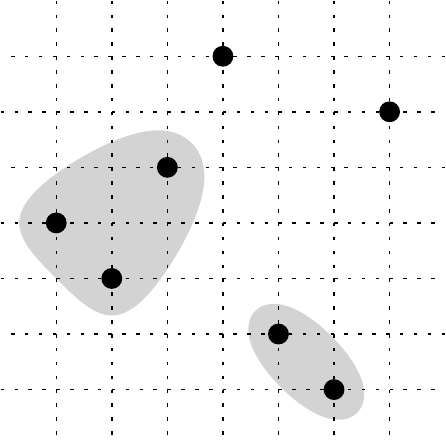}
\end{displaymath}
\caption{An example of LR-inflation: $2413\langle213, 21\rangle = 4357216$.}
\label{fig:LR-inflation}
\end{figure}

\begin{definition}
We say that a permutation class $C$ is \emph{closed under LR-inflations} if for every $\pi\in C$
with $k$ LR-minima, and for every $k$-tuple $\sigma_1,\dotsc,\sigma_k$ of permutations from $C$, the
LR-inflation $\pi\langle\sigma_1,\dotsc,\sigma_k\rangle$ belongs to~$C$. The \emph{closure of $C$
under LR-inflations}, denoted $\LRcl{C}$, is the smallest class which contains $C$ and is closed
under LR-inflations.
\end{definition}

Recall that one can characterize inflation-closed classes by a basis that consists of simple
permutations. We can derive a similar characterization in the case of classes closed under
LR-inflations. We say that a permutation is \emph{LR-simple} if it cannot be obtained by
LR-inflations except for the trivial ones. Using the same arguments, it is easy to see that a
permutation class is closed under LR-inflations if and only if every permutation in its basis is
LR-simple.

\subsection{LR-splittability}
\label{sec:LR-splittability}
We aim to define a stronger version of splittability that would help us connect the properties of
permutation classes and their LR-closures. A natural way to do that is to consider an operation
similar to the regular merge, with LR-minima being shared between both parts.

\begin{definition}
We say that a permutation $\pi$ is a \emph{LR-merge} of permutations $\tau$ and $\sigma$, if its non
LR-minimal elements can be partitioned into two disjoint sequences, such that one of them is,
together with the sequence of LR-minima of $\pi$, an occurrence of $\tau$, and the other is,
together with the sequence of LR-minima of $\pi$, an occurrence of $\sigma$. For two permutation
classes $A$ and $B$, we write $A \LRdot B$ for the class of all LR-merges of a permutation from $A$
with a compatible permutation from $B$. Trivially, $A \LRdot B$ is again a permutation class.
\end{definition}

\begin{figure}[h!]
\centering
\begin{displaymath}
\mathfigure{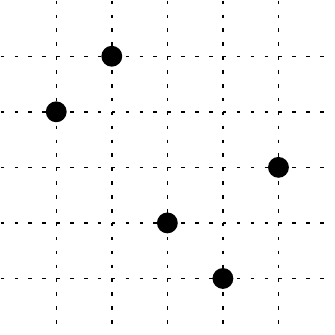} \qquad \mathfigure{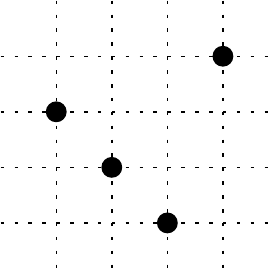} \quad \longrightarrow \quad
\mathfigure{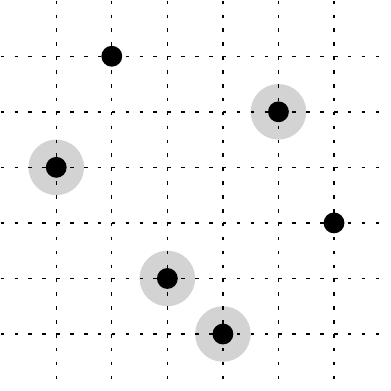}
\end{displaymath}
\caption{For example one possible LR-merge of 45213 and 3214 is the permutation 462153. The
corresponding embedding of 3214 is indicated.}
\label{fig:LR-merge}
\end{figure}

Note that we can also look at LR-merges as a special red-blue colorings of permutations in which the
LR-minima are both blue and red at the same time. Naturally we can use this definition of LR-merge
to define LR-splittability in the same way that the concept of regular merge gives rise to the
definition of splittability.

\begin{definition}
We say that a multiset of permutation classes $\lbrace P_1, \ldots, P_m\rbrace$ forms a
\emph{LR-splitting} of a permutation class $C$ if $C \subseteq P_1 \LRdot \cdots \LRdot P_m$. We
call $P_i$ the \emph{parts} of the LR-splitting. The LR-splitting is \emph{nontrivial} if none of
its parts is a superset of $C$, and the LR-splitting is \emph{irredundant} if no proper submultiset
of $\lbrace P_1, \ldots, P_m\rbrace$ forms an LR-splitting of $C$. A permutation class $C$ is then
\emph{LR-splittable} if $C$ admits a nontrivial LR-splitting.
\end{definition}

Clearly, every LR-splittable class is splittable. Moreover, some properties of LR-splittability are
analogous to the properties of splittability, as shown by the following lemma. We omit the proof as
it uses the very same (and easy) arguments as the proof of Lemma \ref{lemma:splittable conditions}.

\begin{lemma}
\label{lemma:LR-splittable conditions}
For a class $C$ of permutations, the following properties are equivalent:
\begin{enumerate}[(a)]
  \item $C$ is LR-splittable.
  \item $C$ has a nontrivial LR-splitting into two parts.
  \item $C$ has an LR-splitting into two parts, in which each part is a proper subclass of $C$.
  \item $C$ has a nontrivial LR-splitting into two parts, in which each part is a principal class.
\end{enumerate}
\end{lemma}

Now we can state some of the results connecting splittability and LR-splittability of permutation
classes and their LR-closures.
\begin{proposition}
\label{prop:split of LR-closed}
Let $C$ be a permutation class that is closed under $LR$ inflations. Then $C$ is splittable if and
only if $C$ is LR-splittable.
\end{proposition}
\begin{proof}
Trivially, LR-splittability implies splittability since we can take the corresponding red-blue
coloring and simply assign an arbitrary color to each of the LR-minima. Now suppose that $C$ admits
splitting $\lbrace D, E \rbrace$ for some proper subclasses $D$ and $E$. We aim to prove that also
$C \subseteq D \LRdot E$. Let us first show that $C$ contains a permutation $\tau$ that belongs
neither to $D$ nor to~$E$. From the definition of splittability, there are permutations $\tau_D \in
C \setminus D$ and $\tau_E \in C \setminus E$. Define $\tau$ as the LR-inflation of $\tau_D$ with
$\tau_E$, which clearly lies outside both subclasses $D$ and $E$.

Let us suppose that there is some $\pi \in C$ not belonging to $D \LRdot E$ , i.e., there is no
red-blue coloring of $\pi$ which proves it is an LR-merge of a permutation $\alpha \in D$ and a
permutation $\beta \in E$. Let $\pi'$ be the permutation created by inflating each LR-minimum of
$\pi$ with $\tau$. Since $\pi'$ belongs to $C$, it has a regular red-blue coloring with the
permutation corresponding to the red elements $\pi'_R \in D$ and the permutation corresponding to
the blue elements $\pi'_B \in E$. However there must be both colors in each block created by
inflating a LR-minimum of $\pi$ with $\tau$, and therefore there is a valid red-blue coloring of
$\pi$ that assigns both colors to the LR-minima.
\end{proof}

Finally, we want to show that, under modest assumptions, the LR-splittability of a permutation class
implies the LR-splittability (and thus the splittability) of its LR-closure.

\begin{proposition}
\label{prop:LR-split split}
If $C$, $D$ and $E$ are permutation classes satisfying $C \subseteq D \LRdot E$, then $\LRcl{C}
\subseteq \LRcl{D}\LRdot\LRcl{E}$. Consequently, if neither $\LRcl{D}$ nor $\LRcl{E}$ contain the
whole class $C$, then its closure $\LRcl{C}$ is LR-splittable into parts $\LRcl{D}$ and $\LRcl{E}$.
\end{proposition}
\begin{proof}
We will inductively construct a valid red-blue coloring which proves that $\LRcl{C} \subseteq
\LRcl{D} \LRdot \LRcl{E}$. First, any permutation in $\LRcl{C}$ that cannot be obtained from shorter
permutations using LR-inflations must belong to $C$ and we simply use the red-blue coloring that
witnesses the inclusion $C\subseteq D\LRdot E$.

Now take $\pi \in \LRcl{C}$ that can be obtained by LR-inflation from shorter permutations as $\pi =
\alpha\langle\beta_1, \ldots, \beta_k\rangle$. We can already color the permutation $\alpha$ and all
the permutations $\beta_i$ and we construct a coloring of $\pi$ in the following way: color the
inflated blocks $\beta_i$ according to the coloring of $\beta_i$ and the remaining uninflated
elements of $\alpha$ get the color according to the coloring of $\alpha$. It remains to show that
the permutation $\pi_R$ corresponding to the red elements of $\pi$ belongs to $\LRcl{D}$ and the
permutation  $\pi_B$ corresponding to the blue elements of $\pi$ belongs to $\LRcl{E}$. Since the
LR-minima of $\alpha$ are both red and blue, the permutation $\pi_R$ is an LR-inflation of the red
elements of $\alpha$ by the red elements of the permutations $\beta_i$. All these permutations
belong to $\LRcl{D}$ and thus their LR-inflation also belongs to $\LRcl{D}$. Using the very same
argument we can show that $\pi_B$ belongs to $\LRcl{E}$.

It remains to show that the splitting of $\LRcl{C}$ into $\LRcl{D}$ and $\LRcl{E}$ is nontrivial.
However that follows from the assumption that neither $\LRcl{D}$ nor $\LRcl{E}$ contain the whole
class~$C$.
\end{proof}

\subsection{LR-amalgamability}

Similarly to the situation with LR-splittability we want to describe a property of permutation
classes which would imply 1-amalgamability of their respective LR-closures.

\begin{definition}
We say that a permutation class $C$ is \emph{$LR$-amalgamable} if for any two permutations $\tau_1,
\tau_2 \in C$ and any two mappings $f_1$ and $f_2$, where $f_i$ is an embedding of the singleton
permutation into $\tau_i$ and its image is not an LR-minimum of $\tau_i$, there is a permutation
$\sigma \in C$ and two mappings $g_1$ and $g_2$ such that $g_i$ is an embedding of $\tau_i$ into
$\sigma$, $g_1 \circ f_1 = g_2 \circ f_2$ and moreover $g_i$ preserves the property of being a
LR-minimum.
\end{definition}

Observe that LR-amalgamability does not imply 1-amalgamability since it does not guarantee
1-amalgamation over LR-minima and conversely, 1-amalgamability does not imply LR-amalgamability
because it may not preserve the property of being an LR-minimum. However, we can at least prove that
LR-amalgamability implies 1-amalgamability for classes that are closed under LR-inflations. Recall
that we actually derived equivalence between LR-splittability and splittability in Proposition
\ref{prop:split of LR-closed}. 

\begin{lemma}
\label{lemma:amalgamability of LR-closed}
Let $C$ be a permutation class that is closed under $LR$ inflations. If $C$ is LR-amalgamable then
$C$ is also 1-amalgamable.
\end{lemma}
\begin{proof}
Let $\pi_1$ and $\pi_2$ be arbitrary permutations from $C$ and $f_1$, $f_2$ embeddings of the
singleton permutation into $\pi_1$ and $\pi_2$ respectively. If neither of the images of $f_1$ and
$f_2$ is an LR-minimum of the respective permutation we obtain their 1-amalgamation directly since
$C$ is LR-amalgamable.

Now we can assume without loss of generality that the single element in the image of $f_1$ is a
LR-minimum of $\pi_1$. We can create the resulting 1-amalgamation by simply inflating this
LR-minimum by the permutation $\pi_2$. It is then easy to derive the mappings $g_1$ and $g_2$ that
show it is the desired 1-amalgamation.
\end{proof}

We conclude this section by relating LR-amalgamability of a permutation class and 1-amalgamability
of its LR-closure.

\begin{proposition}
\label{prop:LR-am 1-am}
If a permutation class $C$ is LR-amalgamable then its LR-closure $\LRcl{C}$ is LR-amalgamable and
thus also 1-amalgamable.
\end{proposition}
\begin{proof}
Let $\pi_1, \pi_2 \in \LRcl{C}$ be permutations and $f_1, f_2$ embeddings of the singleton
permutation, $f_i$ into $\pi_i$ such that the image of $f_i$ avoids the LR-minima of $\pi_i$. We aim
to prove by induction on the length of $\pi_1$ and $\pi_2$ that there is a corresponding
LR-amalgamation of $\pi_1$ and $\pi_2$. Consider two cases. If neither of the two permutations
$\pi_1$ and $\pi_2$ can be obtained as an LR-inflation of a shorter permutation then they both 
belong
to $C$. And since $C$ itself is LR-amalgamable they have a desired LR-amalgamation that belongs to
$C$.

Without loss of generality we can now assume that $\pi_1$ can be obtained by LR-inflations as $\pi_1
= \alpha \langle \beta_1, \ldots, \beta_k \rangle$ where the permutations $\alpha, \beta_1, \ldots,
\beta_k$ are all strictly shorter than~$\pi_1$. Again we consider two separate cases. First, assume
that the image of the embedding $f_1$ lies inside the block corresponding to the $j$-th inflated
LR-minimum of $\alpha$, which is order-isomorphic to $\beta_j$. From induction we get a
LR-amalgamation $\sigma$ of $\beta_j$ and $\pi_2$ for the embeddings $f_1'$ and $f_2$, where $f_1'$
is the embedding $f_1$ restricted to the inflated block of $\beta_j$. Observe that the permutation
$\alpha \langle \beta_1, \ldots, \beta_{j-1}, \sigma, \beta_{j+1}, \ldots, \beta_k \rangle$ is
precisely the LR-amalgamation of $\pi_1$ and $\pi_2$ we were looking for.

Finally we have to deal with the situation when the image of the embedding $f_1$ lies outside of the
blocks corresponding to the inflated LR-minima of $\pi_1$. We can obtain from induction a
LR-amalgamation $\sigma$ of $\alpha$ and $\pi_2$ for the embeddings $f_1''$ and $f_2$, where $f_1''$
is the embedding $f_1$ restricted to the permutation $\alpha$. Let $g_1$ be the corresponding
embedding of $\alpha$ into $\sigma$ that preserves the LR-minima. We construct the desired
LR-amalgamation of $\pi_1$ and $\pi_2$ in the following way: take $\sigma$ and for every LR-minimum
of $\alpha$ inflate its image under $g_1$ with the corresponding permutation $\beta_i$. The
resulting permutation is clearly a 1-amalgamation of $\pi_1$ and $\pi_2$, and it also preserves the
LR-minima.

Lemma \ref{lemma:amalgamability of LR-closed} implies that $\LRcl{C}$ is also 1-amalgamable.
\end{proof}

\section{Main result}
\label{sec:main result}
Now we are ready to prove that 1-amalgamability and unsplittability are not equivalent by exhibiting
as a counterexample the LR-closure of $\Av{123}$. First, let us show that this class actually has a
nice basis consisting of only two patterns.

\begin{proposition}
\label{prop:Av123^LR basis}
The class $\Av{1423, 1342}$ is the closure of $\Av{123}$ under LR-inflation.
\end{proposition}

\begin{proof}
First, let us show that any permutation from the LR-closure of $\Av{123}$ avoids both 1423, 1342.
Because both of these patterns contain 123, they would have to be created by the LR-inflations.
However, that is not possible since there is no nontrivial interval in either 1423 or 1342 which
contains the minimum element.

Now, let $\pi$ be a permutation from $\Av{1423, 1342}$.  We will show by induction that this
permutation can be obtained by a repeated LR-inflation of permutations from $\Av{123}$.  If $\pi$
does not contain $123$ the statement is trivially true.  Otherwise, consider the set of the
right-to-left maxima of $\pi$.  We want to show that the remaining elements of $\pi$ can be split
into a descending sequence of intervals.  If this holds then we can get $\pi$ as an LR-inflation of
a $123$-avoiding permutation by permutations order-isomorphic to the intervals.  And by induction
these shorter permutations can be obtained as repeated LR-inflations of 123-avoiding permutations.

Let us show that there is no occurrence of the pattern 132 that maps only the letter 2 on an
RL-maximum. For a contradiction suppose we have such an occurrence and a corresponding embedding $f$
of 132 into $\pi$.  Then there must be an element covered by $\pi_{f(3)}$ since it is not an
RL-maximum, i.e., an element $\pi_k$ such that $k > f(3)$ and $\pi_k > \pi_{f(3)}$.  However, $\pi$
restricted to these four indices would form the pattern 1342. Using the same argument, we can also
show that there is no occurrence of the pattern 132 which maps only the letter 3 on an RL-maximum as
we would get an occurrence of the pattern 1423 together with the RL-maximum covered by the image of
2.

And finally, we conclude by showing that the elements of $\pi$ that are not RL-maxima can indeed be
split into a descending sequence of intervals.  Let $I = \{ i_1, \ldots, i_m \}$ be the index set of
the RL-maxima of $\pi$ and furthermore define $i_0 = 0$ and $\pi_0 = n+1$.  Let us represent the
remaining elements of $\pi$ as a set $A$ of $n - m$ points on a plane

\begin{displaymath}
A = \{(i, \pi_i) \mid \mbox{$\pi_i$ is not an RL-maximum of $\pi$}\}.
\end{displaymath}

We define a partition of $A$ into sets $A_{j,k}$ for any $1 \leq j < k \leq m$

\begin{displaymath}
A_{j,k} = \{ (x,y) \mid (x, y) \in A \mbox{ and } i_{j-1} < x < i_j \mbox{ and } \pi_{i_k} < y <
\pi_{i_{k-1}}\}.
\end{displaymath}

\begin{figure}
\centering
\centeredfigure[scale=0.9]{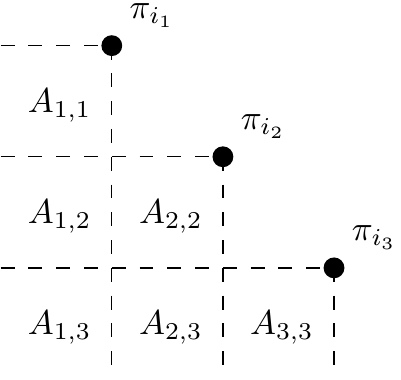} \hspace*{25pt} \centeredfigure[scale=0.9]{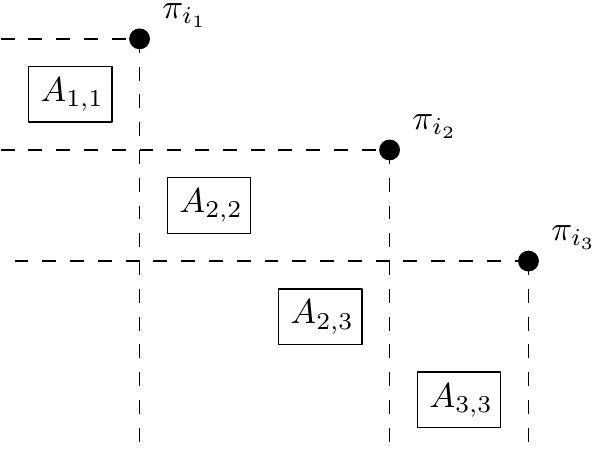}
\caption{Partition of a general permutation with 3 RL-maxima into the sets $A_{j,k}$ and an example
how the
non-empty sets might look for some $\pi \in \Av{1423, 1342}$.}
\label{fig:sets-Ajk}
\end{figure}

For any $j$, $k$ and $l$, every element of $A_{j,k}$ is larger than all the elements of $A_{j+1,l}$
in the second coordinate since otherwise we would get a 132 occurrence with the letter 3 mapped to
$\pi_{i_j}$.  Similarly for any $j$, $k$ and $l$, every element of $A_{j,k}$ is to the left of all
the elements of $A_{l,k+1}$ as otherwise we would get a 132 occurrence with the letter 2 mapped to
$\pi_{i_k}$.  This transitively implies that all non-empty sets $A_{j,k}$ correspond to a sequence
of descending intervals.
\end{proof}

In order to show that $\Av{1423, 1342}$ is splittable, we shall first prove the LR-splittability of
$\Av{123}$ and then apply the results we have obtained in Subsection \ref{sec:LR-splittability}.

\begin{lemma}
\label{lemma:123-splittable}
The class $\Av{123}$ is LR-splittable, and more precisely, it satisfies
\[\Av{123} \subseteq \Av{463152} \LRdot \Av{463152}.\]
\end{lemma}
\begin{proof}
Let $\pi$ be a permutation from $\Av{123}$. Clearly $\pi$ is a merge of two descending sequences,
its LR-minima and the remaining elements. The idea is to decompose the non-minimal elements into
runs such that for every run there is a specific LR-minimum covering each element of the run but
covering none from the following run.  This can be done easily by the following greedy algorithm. In
one step of the algorithm, let $\pi_i$ be the first non-minimal element which was not used yet and
let $j$ be the maximum integer such that $\pi_j$ is an LR-minimum covering $\pi_i$. The next run
then consists of all non-minimal elements starting from $\pi_i$ that are covered by $\pi_j$.

We color each run blue or red such that adjacent runs have different colors.  We obtained a red-blue
coloring of the non-minimal elements and it only remains to check whether the monochromatic
permutations form a proper subclass of $\Av{123}$. Observe that the first elements of two adjacent
runs cannot be covered by a single LR-minimum, which implies that two elements from different
non-adjacent runs cannot be covered by a single LR-minimum. By this observation, in the
monochromatic permutations $\pi_B$ and $\pi_R$ any two elements covered by the same LR-minimum must
belong to the same run.

\begin{figure}
\centering
\centeredfigure{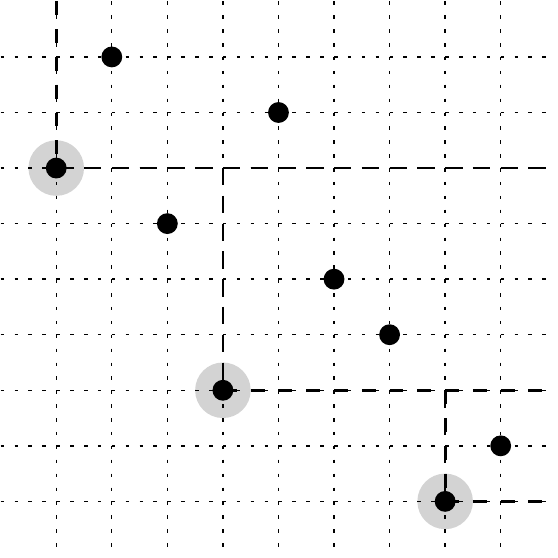}
\caption{For example the 123-avoiding permutation 796385412 with the non-minimal elements split
into three different runs.}
\label{fig:av123-lrsplit}
\end{figure}

We claim that a monochromatic copy of the pattern $463152 \in \Av{123}$ can never be created this
way. Assume for contradiction that there is a permutation $\pi \in \Av{123}$ on which the algorithm
creates a monochromatic copy of 463152 and let $f$ be the corresponding embedding of 463152 into
$\pi$. Observe that every LR-minimum of 463152 is covering some other element and therefore $f$ must
preserve the property of being an LR-minimum, otherwise we would get an occurrence of the pattern
123. Following our earlier observations, the elements $\pi_{f(6)}$, $\pi_{f(5)}$ and $\pi_{f(2)}$
must fall into the same run since $\pi_{f(5)}$ shares LR-minima with both of the other two elements.
And because elements of the same run are covered by a single LR-minimum, there is an LR-minimum
$\pi_i$ covering $\pi_{f(6)}$ and $\pi_{f(2)}$. However, $\pi_i$ must then also cover $\pi_{f(3)}$
which contradicts the fact that $\pi_{f(3)}$ itself is an LR-minimum of~$\pi$.
\end{proof}

\begin{corollary}
\label{cor:123-closure split}
The class $\Av{1423, 1342}$ is splittable.
\end{corollary}
\begin{proof}
In the previous Lemma \ref{lemma:123-splittable} we showed that $\Av{123}$ is LR-splittable, more
precisely that $\Av{123} \subseteq \Av{463152} \LRdot \Av{463152}$ . Since the permutation 463152 is
LR-simple, we get the splittability of $\LRcl{\Av{123}}$ from Proposition \ref{prop:LR-split split}.
Finally, owing to Proposition~\ref{prop:Av123^LR basis}, we know that $\LRcl{\Av{123}}$ and
$\Av{1423, 1342}$ are in fact identical.
\end{proof}

Our final task is to show that $\Av{1423, 1342}$ is 1-amalgamable by proving the LR-amalgamability
of $\Av{123}$. In order to do that we will use the following result which is due to Waton
\cite{Waton2007}. Note that Waton in fact proved the equivalent claim for parallel lines of positive
slope and the permutation class $\Av{321}$.

\begin{proposition}[Waton \cite{Waton2007}]
\label{prop:waton lines}
The class of permutations that can be drawn on any two parallel lines of negative slope is
$\Av{123}$.
\end{proposition}

\begin{lemma}
The class $\Av{123}$ is LR-amalgamable.
\end{lemma}
\begin{proof}
Fix arbitrary two parallel lines of negative slope in the plane. Let $\pi_1$ and $\pi_2$ be
permutations avoiding $123$ and $f_1$ and $f_2$ be mappings where $f_i$ is an embedding of the
singleton permutation into $\pi_i$ and its image is not an LR-minimum of $\pi_i$. According to
Proposition \ref{prop:waton lines} both $\pi_1$ and $\pi_2$ can be drawn from our fixed parallel
lines. Fix sets of points $A_1$ and $A_2$ which lie on these lines whose corresponding respective
permutations are $\pi_1$ and $\pi_2$. Moreover, we can choose the sets such that the elements in the
images of $f_1$ and $f_2$ share the same coordinates. Otherwise we could translate one of the sets
in the direction of the lines to align these two points. Finally, if a point $x \in A_1$ and a point
$y \in A_2$ share one identical coordinate we can move $x$ a little bit in the direction of the
lines without changing the permutation corresponding to the set $A_1$.

We may easily see that the permutation corresponding to the union $A_1 \cup A_2$ with the natural
mappings of $\pi_1$ and $\pi_2$ is the desired LR-amalgamation of $\pi_1$ and $\pi_2$.
\end{proof}

\begin{figure}
\centering
\begin{displaymath}
\mathfigure{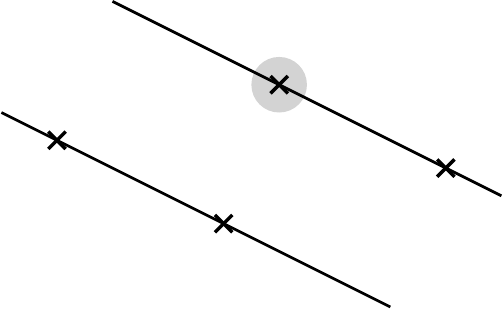} \qquad \mathfigure{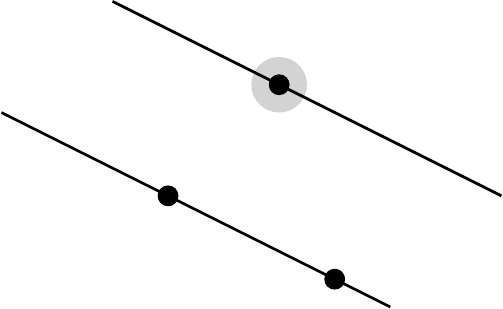} \quad \longrightarrow \quad
\mathfigure{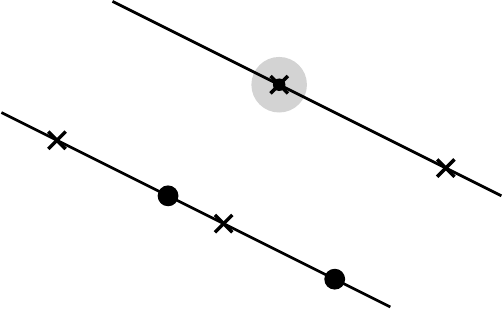}
\end{displaymath}
\caption{Example of two permutations 3142 and 231 drawn from two parallel lines with highlighted
embeddings of the singleton permutation and their LR-amalgamation 532614.}
\label{fig:LR-amalgamation}
\end{figure}
Applying Proposition \ref{prop:LR-am 1-am}, we get the desired result that the class $Av(1423,
1342)$ is indeed 1-amalgamable.

\begin{corollary}
\label{cor:123-closure am}
The class $\Av{1423, 1342}$ is 1-amalgamable.
\end{corollary}

\section{Further directions}

Using our results about LR-inflations, we proved that a single class $\Av{1423, 1342}$ is both
1-amalgamable and splittable. Naturally, the same holds for its three symmetrical classes, i.e.
$\Av{3241, 2431}$, $\Av{4132, 4213}$ and $\Av{2314, 3124}$, since both splittability and
1-amalgamability is preserved when looking at the reverses or complements of the permutations.
However, the question remains whether these results can be used to find more classes that are both
1-amalgamable and splittable or even infinitely many such classes. It would be particularly
interesting to find other such classes with small basis.

Our method of obtaining a splittable 1-amalgamable class was based on the notion of LR-inflations,
and the related concepts of LR-amalgamations and LR-splittings. These notions can be generalized to
a more abstract setting as follows: suppose that we partition every permutation $\pi$ into
`inflatable' and `non-inflatable' elements, in such a way that for any embedding of a permutation
$\sigma$ into $\pi$, the non-inflatable elements of $\sigma$ are mapped to non-inflatable elements
of~$\pi$. We might then consider admissible inflations of $\pi$ (in which only the inflatable
elements can be inflated), admissible splittings of $\pi$ (which are based on two-colorings in which
each inflatable element receives both colors), as well as admissible amalgamations (where we
amalgamate by identifying non-inflatable elements, and the amalgamation must preserve the inflatable
elements of the two amalgamated permutations). In this paper, we only considered the special case
when the inflatable elements are the LR-minima; however, the main properties of LR-inflations,
LR-splittings and LR-amalgamations extend directly to the more abstract setting.

\bibliography{mybib}

\end{document}